\documentclass[12pt]{article}
\usepackage[utf8]{inputenc}
\usepackage{amsmath}             
\usepackage{amssymb}             
\usepackage{amsfonts}            
\usepackage{mathrsfs}            
\usepackage{amsthm}
\usepackage{mathtools}
\usepackage{commath}
\usepackage{bm}
\usepackage{graphicx}
\usepackage{geometry}
\usepackage{float}
\usepackage{listings}
\usepackage{subfigure}
\usepackage{multirow}
\usepackage{diagbox}
\usepackage[export]{adjustbox}
\usepackage[all]{xy}
\usepackage{fancyhdr}
\usepackage{dsfont}
\usepackage{hyperref}
\usepackage{enumitem}
\usepackage{stmaryrd}
\usepackage{tikz-cd}
\DeclareMathAlphabet{\mathbbb}{U}{bbold}{m}{n}

\theoremstyle{definition}
\newtheorem{defi}{Definition}[section]

\theoremstyle{plain}
\newtheorem{theo}[defi]{Theorem}
\newtheorem{lemma}[defi]{Lemma}

\newtheorem{prop-defi}[defi]{Proposition-Definition}

\newtheorem*{fact}{Fact}
\theoremstyle{remark}
\newtheorem*{rmk}{Remark}

\title{Fields of Fractions in Rigid Geometry}
\author{Jiahong Yu\thanks{Academy of Mathematics and Systems Science, Chinese Academy of Sciences} \thanks{Morningside Center of Mathematics, CAS}}

\def\calE{\mathcal{E}}

\def\calO{{\mathcal{O}}}

\begin{document}

\maketitle

\begin{abstract}
Let $A$ be an affinoid integral domain over a non-Archimedean field $K$, and let $L$ be its field of fractions. We prove that the normalization of $A$ can be reconstructed from $L$ by taking the intersection of all maximal discrete valuation subrings. As a corollary, taking the field of fractions induces a fully faithful functor from the category of normal affinoid integral domains over $K$ to the category of field extensions of $K$. This provides another $p$-adic analogue of the Riemann Hebbarkeitssatz.
\end{abstract}

\section{Introduction}

In the theory of algebraic geometry, the field of fractions $F$ of an integral domain $A$ serves as the fundamental object of birational geometry, capturing the properties that are invariant under birational equivalences. By passing to $F$, one can classify varieties up to their `generic' behavior, abstracting away the complexities of specific singularities. However, this abstraction comes with a significant loss of information: the field $F$ itself is `model-blind.' It cannot reconstruct the integral domain $A$ from its field of fractions $F$.

In this paper, we show that in the non-Archimedean setting, the field of fractions carries enough information to reconstruct the analytic structure on its own up to a finite cover. Explicitly, let $K$ be a complete non-Archimedean field, we will prove the following theorem:

\begin{theo}[{Theorem \ref{theo:reconst2}}]\label{theo:intersection of discrete val}
    Let $A$ be an affinoid $K$-algebra whose underlying ring is an integral domain. Let $\mathcal{E}$ be the field of fractions of $A$ and $B \subseteq \mathcal{E}$ be the normalization of $A$. Denote by $\mathbb{X}(\mathcal{E})$ the set of all discrete valuations of $\mathcal{E}$. Then
    \[
    \bigcap_{v \in \mathbb{X}(\calE)} O_v = \left\{
    \begin{matrix}
    B^\circ & \text{if $K$ is discretely valued} \\
    B & \text{if $K$ is not discretely valued}
    \end{matrix}
    \right.
    \]
    where $O_v$ is the valuation ring of $v$, and $B^\circ$ is the subring of power-bounded elements of $B$.
\end{theo}

Obviously, the construction in Theorem \ref{theo:intersection of discrete val} is functoral. Hence, as a corollary, we have the following:

\begin{theo}[{Theorem \ref{theo: proved ration implies algebra}}]\label{theo:ration implies algebraic}
    Let $A$, $B$ be affinoid integral domains over $K$, with their fields of fractions denoted by $\mathcal{E}$ and $\mathcal{F}$ respectively. If $B$ is normal, then any $K$-algebra homomorphism $f: \mathcal{E} \to \mathcal{F}$ satisfies $f(A) \subseteq B$.
\end{theo}

We explain the picture of Theorem \ref{theo:ration implies algebraic} in the context of rigid geometry. In rigid geometry, we regard a homomorphism between two affinoid algebras as a morphism between their corresponding rigid spectra (denoted by $\mathrm{Sp}(-)$ in what follows). Therefore, under the notation of the above theorem, the $k$-algebra homomorphism $f:\mathcal{E}\to \mathcal{F}$ can be regarded as a `meromorphic mapping' from the spectrum $\mathrm{Sp}(B)$ to $\mathrm{Sp}(A)$. Meanwhile, the spectrum of an affinoid algebra is usually viewed as a bounded analytic domain (this point differs from algebraic geometry). Thus, Theorem 1.2 can be seen as a $p$-adic analogue of the following fact in the complex analytic geometry.

\begin{fact}[Riemann Hebbarkeitssatz]
    Suppose $X$ is a connected smooth complex analytic space, $Y \subseteq X$ is an analytic closed subspace of $X$, then any bounded complex analytic function on $U = X - Y$ can be analytically continued to $X$.
\end{fact}

In fact, another $p$-adic analogue of the above fact has been constructed in \cite{1976}. Here we give a more `algebraic' analogue. However, we do not know the relationship between these two analogues at the moment.

Finally, we introduce some applications of Theorem \ref{theo:intersection of discrete val}. From now on, for an abstract ring $A$ and a prime ideal $\mathfrak{p} \subseteq A$, denote by $\kappa(\mathfrak{p})$ the residue field (i.e., the field of fractions of the integral domain $A/\mathfrak{p}$). The first application is a criterion for finiteness of homomorphisms between affinoid algebras:

\begin{theo}[Theorem \ref{cor: max affinoid of finitely generated field proof}]\label{theo: finite critier}
    Let $f: A \to B$ be a homomorphism of affinoid $K$-algebras. Then the following are equivalent:
    \begin{enumerate}
        \item $f$ is finite;
        \item For every minimal prime ideal $\mathfrak{p} \subseteq B$ of $B$, the field extension $\kappa(f^{-1}\big(\mathfrak{p})\big) \subseteq \kappa(\mathfrak{p})$ is a finitely generated field extension.
    \end{enumerate}
\end{theo}

As a second application, we give an alternative proof of the following theorem found in \cite{guo2021prismaticcohomologyrigidanalytic} (using the analogue of the Riemann Hebbarkeitssatz from \cite{1976}) and plays a central role in the construction of the prismatic envelope in the $\mathbb{B}_{\mathrm{dR}}^{+}$-prismatic site.

\begin{theo}[{\cite[Proposition 2.1.1]{guo2021prismaticcohomologyrigidanalytic}}, Theorem \ref{prismatic proof}]\label{theo: prismatic}
    Let $A$, $B$ be affinoid $K$-algebras and $R$ a finitely generated $B$-algebra. Then for a $K$-algebra homomorphism $f: A \to R$, its image $f(A)$ is integral over $B$, that is, each element of $f(A)$ is integral over $B$.
\end{theo}

The discussion in this paper also holds for overconvergent affinoid algebras (a.k.a. dagger algebras) (see the remark of Lemma \ref{lemma: image of complete ring to dis value field}).

\subsection*{Acknowledgements}

We thank Shiji Lyu for his insightful discussions and suggestions.
\newpage
\section*{Notations}
\noindent\begin{itemize}
    \item For a topological ring $A$, denote by $A^\circ$ the subring of power bounded elements. In particular, if $A$ is a normed ring, $A^\circ$ consists of those elements $x$ satisfying that
    \[\sup\{||x^{n}||:n\geq 0\}<\infty,\]
    where $||-||$ is the norm on $A$.

    \item For an integral domain $R$, denote by $\mathrm{Frac}(R)$ the field of fractions of $R$.

    \item `DVR' is the shorten of `discrete valuation ring', i.e., a one dimensional normal noetherian integral domain.

    \item By a discrete valuation $v$ of a field $K$, we always mean an normalized additive valuation. That is, a surjection 
\[ v: K \to \mathbb{Z} \cup \{\infty\} \]
that satisfies the following axioms for all $a, b \in K$:
\begin{enumerate}[label=(\alph*)]
    \item $v(a) = \infty$ if and only if $a = 0$;
    \item $v(ab) = v(a) + v(b)$;
    \item $v(a + b) \ge \min\{v(a), v(b)\}$.
\end{enumerate}

    \item For a field $K$, denote by $\mathbb{X}(K)$ the set of discrete valuations of $K$. In other words, $\mathbb X(K)$ is the set of all subrings $R\subseteq K$ which is a DVR and $\mathrm{Frac}(R)=K$. For $v\in\mathbb X(K)$, denote by $\calO_v$ the valuation ring of $v$.
\end{itemize}
\newpage
\section{First properties of affinoid algebras}

Fix a complete non-Archimedean field $K$. The main object of study in this paper is the affinoid algebras over $K$. In this chapter, we introduce their basic concepts and properties. Recall the following definition:

\begin{defi}[{\cite[Section 2.2, Definition 2]{Bosch_2014}}]
    For any positive integer $n$, the $n$-th \emph{Tate algebra} $T_{K,n}$ of $K$ (abbreviated as $T_n$ in the sequel) is defined as the subring of $K[[z_1,z_2,\dots,z_n]]$ consisting of those elements with convergence radius less than or equal to $1$. More precisely,
    \[
    T_n = \left\{ f = \sum_{\alpha \in \mathbb{Z}_{\geq 0}^n} c_{\alpha} z_1^{\alpha_1} z_2^{\alpha_2} \dots z_n^{\alpha_n} : c_{\alpha} \xrightarrow{|\alpha| = \sum_{i=1}^n \alpha_i \to \infty} 0 \right\}.
    \]
    Also, define $T_0=K$.
\end{defi}

According to \cite[Section 2.2, Proposition 3]{Bosch_2014}, there is a natural norm (Gauss norm) on the Tate algebra $T_n$ which makes it a Banach algebra over $K$. Furthermore, $T_n$ shares many properties similar to the polynomial ring, among which the most important is the following:

\begin{theo}\label{theo: tate is noetherian}
    The Tate algebra $T_n$ is a noetherian ring, and any of its ideals is a closed subset.
\end{theo}

\begin{proof}
    See \cite[Section 2.2, Proposition 14]{Bosch_2014} for the proof of noetherian property of $T$ and \cite[Section 2.3, Corollary 8]{Bosch_2014} for the proof of the closeness of ideals.
\end{proof}

In particular, for an ideal $I \subseteq T_n$, the quotient ring $A = T_n / I$ is a Noetherian Banach $K$-algebra. This leads to the following definition:

\begin{defi}[{\cite[Section 3.1, Definition 1]{Bosch_2014}}]
    An affinoid $K$-algebra is a $K$-algebra that is isomorphic to a quotient of some $T_n$.
\end{defi}

It is important to emphasize that the above definition imposes no topological requirements on $A$. In fact, the following theorem shows that for an affinoid $K$-algebra $A$, the topology on $A$ is entirely determined by $A$ itself.

\begin{theo}\label{theo: canonical topology on affd}
    Suppose $m, n$ are two nonnegative integers, $I \subseteq T_n$ is an ideal, and let $A = T_n / I$. Then for any $K$-homomorphism $f: T_m \to A$, $f$ is continuous, where $A = T_n / I$ is endowed with the quotient topology.
\end{theo}

\begin{proof}
    See \cite[Section 3.1, Proposition 20]{Bosch_2014}.
\end{proof}

For any affinoid $K$-algebra $A$ and a surjection $f: T_n \to A$, endow $A$ with the quotient topology. Then Theorem \ref{theo: canonical topology on affd} shows that this topology does not depend on the choice of $n$ or the surjection $f$. Moreover, this topology is actually induced by a Banach algebra structure on $A$.
The following theorem is also important.

\begin{theo}\label{theo: reduce affd are uniform}
    Let $A$ be a reduced affinoid $K$-algebra. Then $A$ is \emph{uniform}. That is, the ring $A^\circ$ is bounded.
\end{theo}

A direct corollary of this Theorem is the following.

\begin{theo}\label{theo: noetherian of integer when base is discrete}
    Assume $K$ is discrete valued. For any reduced affinoid $K$-algebra $A$, $A^\circ$ is noetherian.
\end{theo}

\begin{proof}
    As $K$ is discrete valued, $A$ has a noetherian ring of definition $A_0$. Let $\pi\in\calO_K$ be a uniformizer, then $A=A_0[\frac{1}{\pi}]$. By Theorem \ref{theo: reduce affd are uniform}, there exists a $N\geq 0$ such that $A^\circ\subset \frac{1}{\pi^N}A_0$. Hence, $A^\circ$ is a finitely generated $A_0$-module.
\end{proof}

By Theorem \ref{theo: tate is noetherian}, any affinoid $K$-algebra is Noetherian. However, for our purposes, we need a stronger property.

\begin{theo}
    Let $A$ be an affinoid $K$-algebra. Then $A$ is an excellent ring (see \cite[\href{https://stacks.math.columbia.edu/tag/07QT}{Tag 07QT}]{stacks-project} for the definition).
\end{theo}

\begin{proof}
    See \cite[Remark 3.5.2]{Fresnel_2004}.
\end{proof}

For readers unfamiliar with the property of excellence, we introduce a weaker property of the above theorem (which is already sufficient for our use).
Recall the following definition in commutative algebras.

\begin{defi}
    Let \( R \) be an integral domain. We say that \( R \) is \textbf{N-1} if the integral closure of \( R \) in its \( \mathrm{Frac}(R) \) is finite over $R$. Similarly, \( R \) is said to be \textbf{N-2} if for every finite extension \( E \) of \( \mathrm{Frac}(R) \), the integral closure of \( R \) in \( E \) is finite over \( R \).
\end{defi}

\begin{defi}
    Let \( R \) be a ring. We say that \( R \) is a \textbf{Nagata ring} if \( R \) is Noetherian and for every prime ideal \( \mathfrak{p} \subseteq R \), the quotient ring \( R/\mathfrak{p} \) is N-2.
\end{defi}

The following result plays a key role in our subsequent discussion.

\begin{theo}\label{theo:nagata affd}
    For any affinoid $K$-algebra $A$, the algebra \( A \) is a Nagata ring.
\end{theo}

\begin{proof}
    See \cite[Theorem 3.5.1, Remark 3.5.2]{Fresnel_2004}.
\end{proof}

\section{Proof of the main theorems}

Throughout this section, we fix a complete non-Archimedean field $K$ and denote by  $\calO_K$ the ring of integers of $K$ and $\mathfrak{m}_Ks$ the maximal ideal of $\calO_K$. We first prove Theorem \ref{theo:intersection of discrete val}.

The goal of this section is to prove the theorems mentioned in the introduction. The key observation is the following analogue of \cite[Theorem 13.1]{Kunz_1986}.

\begin{lemma}\label{lemma: image of complete ring to dis value field}
Let $(R, \mathfrak{m})$ be a Henselian local integral domain with $\mathfrak{m} \neq 0$, $S$ be an $R$-algebra such that $(S, \mathfrak{m} S)$ is a Henselian pair. Let $\mathcal{K}$ be a field containing $R$ as a subring, and $v$ be a discrete valuation on $\mathcal{K}$ with valuation ring $\mathcal{O}_v$. Then for any $R$-algebra homomorphism $f: S \to \mathcal{K}$, we have $f(S) \subseteq \mathcal{O}_v$.
\end{lemma}

\begin{proof}
Let $k = R / \mathfrak{m}$, then $k$ is a field, so there exist infinitely many positive integers $n$ such that $n$ is invertible in $k$. For such $n$ and for any $s \in \mathfrak{m} S$, consider the equation
\[
T^n - (1 + s),
\]
since $(S, \mathfrak{m} S)$ is a Henselian pair, it must have a solution in $S$. This shows that for any $s \in \mathfrak{m} S$, there exist infinitely many integers $n > 0$ such that $f(1 + s) = 1 + f(s)$ has an $n$-th root in $\mathcal{K}$. By the property of discrete valuations, $\big(1 + f(s)\big) \in \mathcal{O}_v$, hence $f(s) \in \mathcal{O}_v$.

Since $\mathfrak{m} \neq 0$, fix a nonzero element $t \in \mathfrak{m}$ and regard it as an element of $\mathcal{K}$. Then by the above conclusion, for any $s \in S$ and any positive integer $m$, $t f(s)^m = f(t s^m) \in \mathcal{O}_v$. By the property of discrete valuations, it follows that $f(s) \in \mathcal{O}_v$.
\end{proof}

\begin{rmk}
    Take any $0 \neq \varpi \in \mathfrak{m}_K$. In this paper, we mainly apply the above lemma to a $\varpi$-adically complete algebra (according to \cite[\href{https://stacks.math.columbia.edu/tag/0ALJ}{Tag 0ALJ}]{stacks-project} and \cite[\href{https://stacks.math.columbia.edu/tag/09XJ}{Tag 09XJ}]{stacks-project}, any $\varpi$-adically complete algebra $R$ satisfies that $(R, \mathfrak{m}_K R)$ is a Henselian pair). However, the above theorem can also be applied to more general situations, such as a filtered colimit of a family of $\varpi$-adically complete algebras (\cite[\href{https://stacks.math.columbia.edu/tag/0FWT}{Tag 0FWT}]{stacks-project}). In particular, the above theorem can be applied to dagger algebras (cf. \cite[Section 1]{grocol}).

\end{rmk}

\begin{theo}\label{theo:reconst2}
    Let $A$ be an affinoid $K$-algebra whose underlying ring is an integral domain. Let $\calE=\mathrm{Frac}(A)$ and $B \subseteq \mathcal{E}$ be the normalization of $A$. Then
    \[
    \bigcap_{v \in \mathbb{X}} O_v = \left\{
    \begin{matrix}
    B^\circ & \text{if $K$ is discretely valued} \\
    B & \text{if $K$ is not discretely valued}
    \end{matrix}
    \right.
    \]
    where $O_v$ is the valuation ring of $v$.
\end{theo}

\begin{proof}
    Since $A$ is a Nagata ring (Theorem \ref{theo:nagata affd}), $B$ is affinoid and we can assume $A=B$. By Lemma \ref{lemma: image of complete ring to dis value field}, we have
    \begin{equation}\label{eq: intersection of DVR > A circ}
        A^\circ\subseteq \bigcap_{v\in\mathbb{X}(\calE)}O_v.
    \end{equation} 
    On the other hand, since \( A \) is a normal Noetherian integral domain, by the algebraic Hartogs lemma (cf. \cite[\href{https://stacks.math.columbia.edu/tag/031T}{Tag 031T}]{stacks-project}), it follows that
\[
A = \bigcap_{v \in \mathbb{X}(\mathcal{E}),\, v(A) \geq 0} O_v .
\]
Therefore, \( A \supseteq \bigcap_{v \in \mathbb{X}(\mathcal{E})} O_v \).

If \( K \) is discretely valued, then Theorem \ref{theo: noetherian of integer when base is discrete} shows that \( A^\circ \) is a Noetherian ring. Since \( A^\circ \) is integrally closed in \( A \) and $A$ is normal, \( A^\circ \) is a normal Noetherian integral domain. Hence, again by \cite[\href{https://stacks.math.columbia.edu/tag/031T}{Tag 031T}]{stacks-project},
\[
A^\circ \supseteq \bigcap_{v \in \mathbb{X}(\mathcal{E})} O_v .
\]

If \( K \) is not discretely valued, we only need to show that for every \( v \in \mathbb{X}(\mathcal{E}) \), \( K \subseteq O_v \). In fact, (\ref{eq: intersection of DVR > A circ}) has shown that $O_v \cap K$ is a ring containing $\mathcal{O}_K$, and therefore it can only be either $\mathcal{O}_K$ or $K$. Since $\mathcal{O}_K$ is not a DVR, and $O_v \cap K$ is either a DVR or $K$ itself, it follows that $K = K \cap O_v$, i.e., $O_v \subseteq K$.

Combining the above consequences, the theorem is proved.
\end{proof}

Now we prove Theorem \ref{theo:ration implies algebraic}. We prove a slightly stronger form.

\begin{theo}\label{theo: proved ration implies algebra}
    Let $A$ be an $\mathcal{O}_K$-algebra such that $(A, \mathfrak{m}_K A)$ is a Henselian pair. Let $B$ be an affinoid $K$-algebra which is itself a normal integral domain. Then the image of any $\mathcal{O}_K$-algebra homomorphism $f: A \to \mathrm{Frac}(B)$ is contained in $B$.
\end{theo}

\begin{proof}
    By Theorem \ref{lemma: image of complete ring to dis value field}, for each discrete valuation $v$ of $\mathrm{Frac}(B)$, we have $f(A) \subseteq \mathcal{O}_v$. Therefore, by Theorem \ref{theo:reconst2}, the image of $f$ is contained in $B$.    
\end{proof}

We now prove Theorem \ref{theo: finite critier}, for which we need the following well-known result:

\begin{lemma}\label{lemma: finite reduce to domain}
    Let $f: A \to B$ be a homomorphism of Noetherian rings. Suppose that for every minimal prime ideal $\mathfrak{p}$ of $B$, the induced map $A \to B/\mathfrak{p}$ is finite. Then, $f$ is finite.
\end{lemma}

\begin{proof}
    Since $B$ is Noetherian, the set of its minimal prime ideals $\{\mathfrak{p}_1, \dots, \mathfrak{p}_n\}$ is finite. Let $\mathfrak{N} = \text{nil}(B)$ be the nilradical of $B$. It is well-known that $\mathfrak{N} = \bigcap_{i=1}^n \mathfrak{p}_i$.
    Consider the naturally induced injective homomorphism:$$\phi: B/\mathfrak{N} \hookrightarrow \bigoplus_{i=1}^n B/\mathfrak{p}_i$$By assumption, each $B/\mathfrak{p}_i$ is a finite $A$-module, so their finite direct sum $\bigoplus B/\mathfrak{p}_i$ is also a finite $A$-module. Since $A$ is Noetherian, any submodule of a finitely generated $A$-module is also finitely generated. Thus, $B/\mathfrak{N}$ is a finite $A$-module.

    Because $B$ is Noetherian, the ideal $\mathfrak{N}$ is finitely generated and nilpotent, i.e., $\mathfrak{N}^k = 0$ for some integer $k \ge 1$. We consider the following filtration:$$B = \mathfrak{N}^0 \supseteq \mathfrak{N}^1 \supseteq \dots \supseteq \mathfrak{N}^k = 0$$For each $0 \le j < k$, consider the quotient layers $M_j = \mathfrak{N}^j / \mathfrak{N}^{j+1}$. Since $\mathfrak{N} \cdot M_j = 0$, $M_j$ admits a structure of a $B/\mathfrak{N}$-module. Since $B$ is Noetherian, $\mathfrak{N}^j$ is a finitely generated $B$-module, which implies that $M_j$ is a finitely generated $B/\mathfrak{N}$-module.

    Given that $B/\mathfrak{N}$ is a finite $A$-algebra and $M_j$ is a finite $B/\mathfrak{N}$-module, it follows by the transitivity of finiteness that each $M_j$ is a finite $A$-module.From the short exact sequences:$$0 \to \mathfrak{N}^{j+1} \to \mathfrak{N}^j \to M_j \to 0$$and the fact that the property of being a finite $A$-module is stable under extensions, we conclude by induction on $j$ (from $k-1$ down to $0$) that $\mathfrak{N}^0 = B$ is a finite $A$-module.Thus, $f: A \to B$ is a finite homomorphism.
\end{proof}

\begin{theo}\label{cor: max affinoid of finitely generated field proof}
    Let $f: A \to B$ be a homomorphism of affinoid $K$-algebras. Then the following are equivalent:
    \begin{enumerate}
        \item $f$ is finite;
        \item For every minimal prime ideal $\mathfrak{p} \subseteq B$ of $B$, the field extension $\kappa(f^{-1}\big(\mathfrak{p})\big) \subseteq \kappa(\mathfrak{p})$ is a finitely generated field extension.
    \end{enumerate}
\end{theo}

\begin{proof}
    By Lemma \ref{lemma: finite reduce to domain}, we may assume that both $A$ and $B$ are integral domains and that $f$ is injective. It suffices to prove that $B$ is integral over $A$ (cf. \cite[Section 6.3.5, Theorem 1]{Bosch_1984}). Replacing $B$ by its normalization, we may assume that $B$ is normal.

First, assume that $K$ is non-discrete. Let $\mathcal{K}$ be the algebraic closure of $\mathrm{Frac}(A)$ in $\mathrm{Frac}(B)$; this is a finite extension of $\mathrm{Frac}(A)$ by \cite[\href{https://stacks.math.columbia.edu/tag/037J}{Tag 037J}]{stacks-project}. Let $\overline{A}$ be the integral closure of $A$ in $\mathcal{K}$, which is finite over $A$ because $A$ is Nagata. By Lemma \ref{lemma: image of complete ring to dis value field}, we have $\overline{A}^\circ \subset B$ and hence $\overline{A} \subset B$.

We claim that $\overline{A} = B$. By Theorem \ref{theo:reconst2},
\[
\overline{A} = \bigcap_{v \in \mathbb{X}(\mathcal{K})} \mathcal{O}_v
\]
and
\[
B = \bigcap_{w \in \mathbb{X}(\mathrm{Frac}(B))} \mathcal{O}_w.
\]
It remains to show that for any discrete valuation $v$ on $\mathcal{K}$,
\[
\bigcap_{w} \mathcal{O}_w
\]
where the intersection runs over all discrete valuations on $\mathrm{Frac}(B)$ extending $v$. This is established in Lemma \ref{lemma:int closure of DVR}.

The same argument also applies when $K$ is discrete.
\end{proof}

\begin{lemma}\label{lemma:int closure of DVR}
    Let $O$ be a DVR with field of fractions $K$. Let $L/K$ be a finitely generated field extension. Let $X$ be the subset of $\mathbb{X}(L)$ consists of those valuations that dominates $O$. Then:$$\bigcap_{w \in X} \calO_w = B$$where $B$ is the integral closure of $O$ in $L$.
\end{lemma}

\begin{proof}
    Let $\mathfrak{m}$ be the maximal ideal of $O$, and let $\overline{K}$ be the algebraic closure of $K$ in $L$. By \cite[\href{https://stacks.math.columbia.edu/tag/037J}{Tag 037J}]{stacks-project}, $\overline{K}$ is a finite extension of $K$. Therefore, by Krull–Akizuki's Theorem (\cite[\href{https://stacks.math.columbia.edu/tag/00PG}{Tag 00PG}]{stacks-project}), $B$ is a Dedekind domain. Suppose $\mathfrak{p}$ is a maximal ideal of $B$, then $B_{\mathfrak{p}}$ is a DVR dominating $O$. Thus it suffices to show that the intersection of the valuation rings in $\mathbb{X}(L)$ that dominate the valuation of $B_{\mathfrak{p}}$ is $B_{\mathfrak{p}}$. This reduces to the case $K = \overline{K}$ (i.e., $O = B$).

For any $x \in L \setminus O$, $R := O[x^{-1}] \subseteq L$ is either $O$ itself or isomorphic to the polynomial ring in one variable over $O$. In both cases, $M = (\mathfrak{m}, x^{-1})$ generates a maximal ideal. By \cite[\href{https://stacks.math.columbia.edu/tag/00PH}{Tag 00PH}]{stacks-project}, there exists $w \in \mathbb{X}(L)$ such that $\mathcal{O}_w$ dominates $R_M$. Therefore, $w \in X$ and $x \notin \mathcal{O}_w$. This shows that $O \subseteq \bigcap_{w \in X} \calO_w$. The other inclusion is obvious.
\end{proof}

Finally, we prove Theorem \ref{theo: prismatic}.

\begin{theo}\label{prismatic proof}
    Let $A$, $B$ be affinoid $K$-algebras and $R$ a finitely generated $B$-algebra. Then for a $K$-algebra homomorphism $f: A \to R$, its image $f(A)$ is integral over $B$, that is, each element of $f(A)$ is integral over $B$.
\end{theo}

\begin{proof}
    First, assume that $R$ is an integral domain. At this point, by replacing $A$ and $B$ with their images in $R$, we may assume that $A$ and $B$ are both subrings of $R$. The proof of Theorem \ref{cor: max affinoid of finitely generated field proof} shows that the integral closure $\overline{B}$ of $B$ in $\mathrm{Frac}(R)$ is a finite algebra over $B$, so we can replace $B$ with $\overline{B}$, and thus we assume that $B$ is integrally closed in $\mathrm{Frac}(R)$. Therefore, it suffices to prove that $A \subseteq B$, which is equivalent to $A^\circ \subseteq B$. According to the proof of Theorem \ref{cor: max affinoid of finitely generated field proof}, $A^\circ$ is contained in $\bigcap_{w \in \mathbb{X}(\mathrm{Frac}(R))}$, and the latter is a subring of $B$. Hence, the theorem is proved.

    For the general case, by Step 1 of the proof of \cite[Proposition 2.1.1]{guo2021prismaticcohomologyrigidanalytic}[, we can reduce to the case where $R$ is reduced. Suppose $\{\mathfrak{p}_1, \mathfrak{p}_2, \dots, \mathfrak{p}_r\}$ are the minimal prime ideals of $R$, then $R$ is a subring of 
\[
R' := \prod_{j=1}^{r} R / \mathfrak{p}_j.
\]
Applying the conclusion for integral domains to $R'$ suffices.
\end{proof}

\bibliographystyle{alpha}
\bibliography{ref}
\addcontentsline{toc}{section}{References}

\end{document}